\documentclass[12pt,leqno]{article}
\usepackage{amssymb,amsmath,amsthm}
\usepackage[all]{xy}

\usepackage{graphics}
\usepackage[alphabetic,lite]{amsrefs}
\numberwithin{equation}{subsection}
\usepackage{tikz}
\usepackage[top=2cm,bottom=2cm,left=3cm,right=3cm,marginparsep=0.3cm,marginparwidth=3cm,includefoot]{geometry}
\usepackage{mathtools}
\usepackage{temp}

\usepackage[normalem]{ulem}  % for strikeout \sout
\usepackage{mathrsfs}
\usepackage{accents}
\usepackage{xspace}
\usepackage{marginnote}

\begin{document}

\title{Vanishing of temperate cohomology on complex manifolds}
\author{Pierre Schapira}
\maketitle

\begin{abstract}
Consider a complex Stein manifold $X$ and a subanalytic relatively compact Stein open subset $U$ of $X$. 
We prove the vanishing   on $U$ of the holomorphic temperate cohomology. 
\end{abstract}

{\renewcommand{\thefootnote}{\mbox{}}
\footnote{This research was supported by the  ANR-15-CE40-0007 ``MICROLOCAL''.}
\addtocounter{footnote}{-1}
}

\tableofcontents
\section*{Introduction}

The theory of ind-sheaves and, as a byproduct, the subanalytic topology  on a real analytic manifold $M$ and  the site $\Msa$,  have been introduced in~\cite{KS01} after the construction by M.~Kashiwara in~\cite{Ka84} of the functor $\thom$ of temperate cohomology.

On the site $\Msa$ one easily obtains  various sheaves which would have no meaning in the classical setting such that the sheaves 
 of $\Cinf$-functions or distributions  with temperate growth or the sheaf of Whitney  $\Cinf$-functions. 
Notice that this topology has been  refined in~\cite{GS16} in which  the linear subanalytic topology and the site $\Msal$ were introduced but we shall not use this topology here. 

On a complex manifold $X$, by considering the Dolbeault complex with coefficients in these sheaves, we get various (derived) sheaves of tempered holomorphic functions on the site $\Xsa$ and in particular the sheaf $\Ot_\Xsa$ of  temperate holomorphic functions. 

A natural question is then to prove the  vanishing of the cohomology  of  $\Ot_\Xsa$  on a subanalytic relatively compact Stein open subset of a complex manifold. Many specialists of complex analysis consider the answer to this question as easy or well-known but we have not found any proof of it in the literature, despite the fact of its many applications (see {\em e.g.,}~\cite{Pe17}). The aim of this Note is to provide a short proof to this result, by combining the fundamental and classical vanishing theorem of H\"ormander~\cite{Ho65} 
together with~\cite{KS96}*{Th.~10.5}. 

\noindent{\bf Acknowledgments}
\rm{(i)} We warmly thank Henri Skoda who is at the origin of this paper. Indeed,  this is when discussing together of the problem that I realized that Theorem~\ref{th:main1} could be easily deduced from H\"ormander vanishing theorem. On his side, Henri will soon publish 
a paper \lp see~\cite{Sk20}\rp\, in which he obtains a similar and more precise result to this theorem with a weaker hypothesis, namely without any subanalyticity assumption. 

\spa
\rm{(ii)} We also warmly thank Daniel Barlet who explained me how the result of Lemma~\ref{le:barlet} could easily be deduced from~\cite{GR65}*{Ch.~8, Sect.~C,Th. 8}.
%\end*{acknowledgments}

\section{Review on the subanalytic site (after~\cite{KS01})} 

In this paper, $M$ denotes a real analytic manifold endowed with a distance, denoted  $\dist$, 
 locally Lipschitz equivalent to the Euclidian distance on $\R^n$.  We set 
$\dist(x,\varnothing) = D_M + 1$, where $D_M$ is the diameter of $M$. We also choose a measure $d\lambda$ on $M$ locally isomorphic to the Lebesgue measure on $\R^n$. For a relatively compact open subset $U\subset M$, we shall denote by $\vvert\cdot\vvert_{L^p}$  the $L^p$-norm ($p=1,2,\infty$) for this measure on $U$. Since $U$ is relatively compact, this norm does not depend on the choice of $d\lambda$, up to constants. 

\subsection{Sheaves}
We fix a field $\cor$. 
We denote by $\md[\cor_M]$ the abelian category of sheaves of $\cor$-modules on $M$ and by $\Derb(\cor_M)$ its bounded derived category. References for sheaf theory are made to~\cite{KS90}. 

In particular, we shall use the duality functor
\eqn
&&\RD'_M\eqdot\rhom(\scbul,\cor_M).
\eneqn
A sheaf $F$ on  $M$ is $\R$-constructible if there exists a stratification 
$M=\bigsqcup_\alpha M_\alpha$  by locally closed subanalytic subsets $M_\alpha$ such that $F\vert_{M_\alpha}$ is locally constant of finite rank. We denote by $\mdrc[\cor_M]$ the abelian category of $\R$-constructible sheaves on $M$. Its bounded derived category is equivalent to $\Derb_\Rc(\cor_M)$ the full triangulated subcategory of $\Derb(\cor_M)$ consisting of objects with $\R$-constructible cohomology. The derived categories of  $\R$-constructible sheaves satisfy the formalism of Grothendieck's six operations.

%\subsection{The subanalytic topology}
\subsection{Sheaves on the subanalytic site}

 Let us denote by $\SA(M)$ the family of subanalytic  subsets  of the real analytic manifold $M$. 
 
 This family contains that of  semi-analytic sets and satisfies:
\begin{itemize}
\item
the property of being subanalytic is local on $M$, 
\item
 $\SA(M)$ is stable by finite intersections, finite unions,  complementary in $M$,  closure, interior,
\item
a compact subanalytic set is topologically isomorphic to a finite CW-complex,
 \item 
for $f\cl M\to N$ a morphism,   $Z\in\SA(N)$ and $S\in\SA(M)$,  $\opb{f}(Z)\in\SA(M)$  and 
$f(S)\in\SA(N)$ as soon as  $f$ is proper on $\ol S$.
\end{itemize}
References are made to the pioneering work of Lojasiewicz, followed by those of Gabrielov and Hironaka and, for a modern treatment, to the paper~\cite{BM88} by E.~Bierstone and P.~Milman.

The main property of subanalytic sets are the so called Lojasiewicz inequalities which play an essential role here.

\begin{lemma}[{Lojasiewicz inequalities}]\label{le:lojaleq}
Let $U=\bigcup_{j\in J}U_j$ be a finite covering in $\Op_\Msa$. 
Then there exist a constant $C>0$ and a positive integer $N$ such that  
\eq\label{eq:lojaleq}
&&\dist(x,M\setminus U)^N\leq C\cdot(\max_{j\in J}  \dist(x,M \setminus U_j)).
\eneq
\end{lemma}

The subanalytic site $\Msa$ associated to a real analytic manifold $M$ is defined as follows.
\begin{itemize}
\item
$\Op_\Msa$ is the category of relatively compact  subanalytic open subsets of $M$,
\item
the coverings are the finite coverings, meaning that a family $\{U_i\}_{i\in I}$ of objects of $\Op_\Msa$ is a covering of $U\in\Op_\Msa$ if 
$U_i\subset U$ for all $i$ and there exists a finite subset $J\subset I$ such that $\bigcup_{j\in J}U_j=U$.
\end{itemize}
Hence, we get the abelian category $\md[\cor_\Msa]$ of sheaves on $\Msa$ and its bounded derived category $\Derb(\cor_\Msa)$.

On denotes by $\rho_M\cl M\to\Msa$  the natural morphism of sites. 
As usual, one gets a pair of adjoint functors 
\eq\label{eq:oimrho}
&& \xymatrix@C=6ex
{
\opb{\rho_M}\cl\md[\cor_\Msa]\ar@<0.5ex>[r]&\md[\cor_M]\ar@<0.5ex>[l] \cl\oim{\rho_M}.
}
\eneq
The functor $\oim{\rho_M}$ is fully faithful and the functor $\opb{\rho_M}$ also admits a left adjoint:
\eq\label{eq:eimrho}
&& \xymatrix@C=6ex
{
\eim{\rho_M}\cl\md[\cor_M]\ar@<0.5ex>[r]&\md[\cor_\Msa]\ar@<0.5ex>[l] \cl\opb{\rho_M}.
}
\eneq
For $F\in \md[\cor_M]$,  $\eim{\rho_M}F$ is the sheaf associated with the presheaf $U\mapsto F(\ol U)$. Moreover, the functor $\eim{\rho_M}$ is exact.

Recall that the functor $\oim{\rho_M}$ is exact on $\mdrc[\cor_M]$ and this last category may be considered as a full thick subcategory of  $\md[\cor_M]$
as well as of  $\md[\cor_\Msa]$. Similarly, 
$\Derb_\Rc(\cor_M)$ can be considered  as a full triangulated subcategory of $\Derb(\cor_M)$ as well as of  $\Derb(\cor_\Msa)$.

We shall need the next result, already proved in~\cite{KS01}*{Cor.~4.3.7} in the more general framework of indsheaves. For the reader's convenience, we give a direct proof.
\begin{lemma}\label{le:eimrhooimf}
Let $f\cl M\to N$ be a morphism of manifolds. There is a natural isomorphism of functors
$\opb{f}\eim{\rho_N}\simeq\eim{\rho_M}\opb{f}$. 
\end{lemma}
\begin{proof}
By adjunction, it is enough to check the isomorphism of functors 
\eqn
&&\opb{\rho_N}\oim{f}\simeq\oim{f}\opb{\rho_M}.
\eneqn
Denote  by $\pshopb{\rho_M}$ the inverse image functor for presheaves associated to  the morphism $\rho_M$ and similarly for $\rho_N$. Denote by $(\cdot)^a$ the functor which associates a sheaf to a presheaf. Since direct images commute with 
 $(\cdot)^a$ and $\Op_\Nsa$ is a basis of open subsets of $N$,  it is enough to prove the isomorphism of functors of presheaves on $\Nsa$
 \eqn
&&\pshopb{\rho_N}\oim{f}\simeq \oim{f}\pshopb{\rho_M}.
\eneqn
Let $F\in\md[\cor_\Msa]$ and let $V$ be open in $\Nsa$. Then
\eqn
&&(\pshopb{\rho_N}\oim{f}F)(V)\simeq\oim{f}F(V)\simeq F(\opb{f}V)\simeq (\pshopb{\rho_N}F)(\opb{f}V)\simeq (\oim{f}\pshopb{\rho_M}F)(V).
\eneqn
\end{proof}

Recall that one says after~\cite{GS16} that a sheaf $F$ on $\Msa$ is $\sect$-acyclic if $\rsect(U;F)$ is concentrated in degree $0$ for all $U\in\Op_\Msa$.
\begin{proposition}[{see~\cite{KS01}*{Pro.~6.4.1} and~\cite{GS16}*{Pro.~2.14} }]\label{pro:cnssheaf} 
{\rm (i)} A presheaf $F$ on $\Msa$ is a sheaf as soon as 
 $F(\varnothing)=0$ and, 
for any $U_1,U_2\in\Op_\Msa$, the sequence below is exact:
\eq\label{eq:MV2}
&&0\to F(U_1\cup U_2)\to F(U_1)\oplus F(U_2)\to F(U_1\cap U_2).
\eneq
{\rm (ii)} If moreover, for any $U_1,U_2\in\Op_\Msa$, the sequence
\eq\label{eq:MV2b}
&&0\to F(U_1\cup U_2)\to F(U_1)\oplus F(U_2)\to F(U_1\cap U_2)\to 0
 \eneq
is exact, then $F$ is $\Gamma$-acyclic.
\end{proposition}

\subsection{Some subanalytic sheaves on real manifolds}

\subsubsection*{Classical sheaves}
Let $M$ be a real analytic manifold. One denotes as usual by $\sha_M,  \shc^{\infty}_M, \Db_M,\shb_M$ the sheaves of complex valued real analytic functions, $C^\infty$-functions, distributions and hyperfunctions on $M$. We also denote by $\Omega_M$ the sheaf of real analytic differential forms of maximal degree, by $\ori_M$ the orientation sheaf and by $\shv_M$ the sheaf of real analytic densities, 
$\shv_M=\Omega_M\tens\ori_M$. Finally, we denote by $\shd_M$ the sheaf of finite order differential operators with coefficients in $\sha_M$.

\subsubsection*{Sheaves of temperate functions and distributions}

Let $U\in\Op_\Msa$. Set for short 
\eq\label{eq:notadist}
&&\dist_U(x)=\dist(x,M\setminus U)
\eneq
and denote as usual by $\vvert\cdot\vvert_{L^\infty}$ the sup-norm.
\bnum
\item
One says that $f\in \shc^{\infty}_M(U)$ has
{\it  polynomial growth} 
if there exists $N\geq0$ such that
\eqn
&&\vvert\dist_U(x)^N f(x)\vvert_{L^\infty}<\infty\,.
\eneqn
\item
One says that $f\in \shc^{\infty}_M(U)$ is  temperate 
 if all its derivatives (in local charts) have polynomial growth. 
 \item
One says that a distribution $u\in \Db_M(U)$ is temperate if $u$ extends as a distribution on $M$.
\enum

For  $U\in\Op_\Msa$, denote by 
\begin{itemize}
\item
$\Cinftp_M(U)$ the subspace of $\Cinf_M(U)$ consisting of temperate functions on $U$,
\item
$\Dbt_M(U)$ the space of temperate distributions on $U$.
\end{itemize}
Denote by $\Cinftp_\Msa$ and $\Dbt_{\Msa}$ the presheaves so defined.
Using  Lojasiewicz's inequalities, one  checks that 
these presheaves   are  sheaves on $\Msa$.  Moreover, $\Cinftp_\Msa$ is a sheaf of rings
and  $\Dbt_{\Msa}$  is a sheaf of  $\Cinftp_\Msa$-modules.

The next lemma below is an essential tool for our study.
\begin{lemma}[{see~\cite{Ho83}}]\label{le:hopartition}
Let $Z_0$ and $Z_1$ be two closed subanalytic subsets of $M$. There exists $\psi\in\Cinftp(M\setminus(Z_0\cap Z_1))$ such that $\psi=0$ in a neighborhood of $Z_0$ and $\psi=1$  in a neighborhood of $Z_1$.
\end{lemma}

\begin{lemma}
Any sheaf of $\Cinftp_\Msa$-modules on $\Msa$  is  $\Gamma$-acyclic.
\end{lemma}
\begin{proof}
Apply Lemma~\ref{le:hopartition}  as in~\cite{GS16}*{Prop.~4.18}.
\end{proof}
In particular, the sheaf  $\Dbt_{\Msa}$  is $\Gamma$-acyclic.

We shall also use the sheaf
\eqn
&& \Db^{\tp \vee}_\Msa\eqdot \Dbt_\Msa\tens[\eim{\rho_M}\sha_M] \eim{\rho_M}\shv_M.
\eneqn

\subsubsection*{Sheaf of Whitney functions}
For a closed subanalytic subset $S$ in $M$, 
denote by $\cI^\infty_{M,S}$ the space
of $\mathrm{C}^\infty$-functions  defined on $M$ which vanish up to 
infinite order on $S$. 
In~\cite{KS96}, one  introduced  the sheaf:
\eqn
\C_U\wtens \Cinf_M&:=&V\longmapsto \cI^\infty_{V,V\setminus U}
\eneqn
and showed  that it  uniquely extends  to an exact functor 
\eqn
&&\scbul\wtens \shc^{\infty}_M,\quad \mdrc[\C_M]\to \md[\C_M].
\eneqn
One denotes by $\Cinfw_\Msa$ the sheaf on $\Msa$ given by 
\eqn
&&\Cinfw_\Msa(U)=\sect(M;H^0(\RD'_M\cor_U)\wtens \shc_M^{\infty}), U\in \Op_{\Msa}.
\eneqn
If $\RD'_M\C_U\simeq\C_{\ol U}$, then $\Cinfw_\Msa(U)\simeq \Cinf(M)/
\cI^\infty_{M,\ol U}$ is the space of Whitney functions on
$\ol{U}$. 
It is thus natural to call $\Cinfw_\Msa$ the sheaf of Whitney $\Cinf$-functions on $\Msa$. 

Recall that a $\C$-vector space is of type $\FN$ (resp.\ $\DFN$)  if it is a Fr{\'e}chet-nuclear space (resp.\ dual of a  Fr{\'e}chet-nuclear space).
\begin{proposition}[see~{\cite{KS96}*{Prop.~2.2}}]\label{pro:KS9622}
 Let $U\in\Op_\Msa$.
 There exist natural topologies
of type $\FN$ on $\sect(M; \C_U\wtens \Cinf_M)$ and of type
$\DFN$ on $\sect(U; \Db^{\tp \vee}_\Msa))$ and they
are dual to each other.
\end{proposition}

Note that the sheaf $\oim{\rho_M}\shd_{M}$ does not operate on the sheaves  $\Cinft_\Msa$,
$\Dbt_\Msa$, $\Cinfw_\Msa$ but  $\eim{\rho_M}\shd_{M}$ does.

\begin{remark}
It would have been more natural to consider the cosheaf $U\mapsto\sect(M;\cor_U\wtens\shc_M^{\infty})$. We didn't since cosheaf theory is still not well-established.
\end{remark}

\subsubsection*{Sheaves of temperate $L^2$-functions}
Recall that $M$ is endowed with a measure $d\lambda$ locally equivalent to the Lebesgue measure. 
Denote  by $L^0_M$ the sheaf of measurable functions on $M$  and recall that $\vvert\cdot\vvert_{L^2}$ denotes the $L^2$-norm. Also recall notation~\eqref{eq:notadist}.
For $U$ open and relatively compact in $M$ and $s\in\R_{\geq0}$, set
\eq\label{eq:L2sU}
&&L^{2,s}(U)=\{f\in L^0_M(U); \vvert\dist_U^s(x)f(x)\vvert_{L^2}<\infty\}.
\eneq
Also set
\eq\label{eq:L2tpU}
&&L^{2,\tp}(U)\eqdot\indlim[s\geq0]L^{2,s}(U).
\eneq
Denote by $\shl^{2,\tp}_\Msa$ the presheaf $U\mapsto L^{2,\tp}(U)$ on $\Msa$. 
\begin{lemma}\label{le:L2tpshv}
\bnum
\item
The presheaf $\shl^{2,\tp}_\Msa$ is a   sheaf on $\Msa$.  
\item
The sheaf  $\shl^{2,\tp}_\Msa$  is a $\Cinftp_\Msa$-module and in particular is $\sect$-acyclic.
\enum
\end{lemma}
\begin{proof}
(i) follows immediately from Lojasiewicz's inequalities (Lemma~\ref{le:lojaleq}) and the fact that coverings are finite coverings.

\spa
(ii) Let $\phi\in L^{2,s}(U)$ and let $\theta\in \Cinftp_\Msa(U)$. Then 
$\vvert\dist_U(x)^t\cdot \theta(x)\vvert_{L^\infty}<\infty$ for some $t\geq0$ and we get
\eqn
&&\vvert\dist_U(x)^{t+s}\theta(x)\phi(x)\vvert_{L^2}\leq C\cdot \vvert\dist_U(x)^t\theta(x)\vvert_{L^\infty}\cdot\vvert\dist_U(x)^{s}\phi(x)\vvert_{L^2}<\infty.
\eneqn
Hence, $\theta\phi$ belongs to  $L^{2,s+t}(U)$.
\end{proof}

\begin{proposition}\label{pro:l2Ctpmod}
One has natural monomorphisms   $\Cinftp_\Msa\into \shl^{2,\tp}_\Msa\into\Dbt_\Msa$. 
\end{proposition}
\begin{proof}
The monomorphism $ \shl^{2,\tp}_\Msa\into\Dbt_\Msa$ is obvious.

\spa
(ii) Let $U\in\Op_\Msa$ and  let $\phi\in \Cinftp_\Msa(U)$.
There exists $t\in\R_{\geq0}$ such that 
\eqn
&&\vvert\dist_U(x)^t\phi(x)\vvert_{L^\infty}<\infty.
\eneqn
Hence, we have for some constant $C>0$
\eqn
&&\vvert\dist_U(x)^t\phi(x)\vvert_{L^2}\leq C\cdot \vvert\dist_U(x)^t\phi(x)\vvert_{L^\infty}<\infty
\eneqn
\end{proof}

\subsection{Some subanalytic sheaves on complex manifolds}

\subsubsection*{Temperate holomorphic functions}
Now let $X$ be a {\em complex} manifold of complex dimension $n$. 
One defines the (derived) sheaf of temperate holomorphic functions 
$\Ot_\Xsa\in \Derb(\C_\Xsa)$ as the Dolbeault complex with coefficients  in  
$\Cinftp_{\Xsa}$.
In other words
\eq\label{eq:dolbOt1}
\Ot_\Xsa&\eqdot&0\to \shc_\Xsa^{\infty,\tp (0,0)}\to[\ol\partial]\cdots\to[\ol\partial]\shc_\Xsa^{\infty,\tp (0,n)}\to0.
\eneq
If $n>1$, this object is no more concentrated in degree $0$. 

Also consider the object $\tw\OO^\tp_\Xsa\in \Derb(\C_\Xsa)$
\eq\label{eq:dolbOt3}
\tw\OO^\tp_\Xsa&\eqdot&0\to \Db_\Xsa^{\tp(0,0)}\to[\ol\partial]\cdots\to[\ol\partial]\Db_\Xsa^{\tp(0,n)}\to0.
\eneq

\begin{theorem}[{see~\cite{KS96}*{Th.~10.5}}]\label{th:ks96}
The natural morphism $\Ot_\Xsa\to\tw\OO^\tp_\Xsa$ is an isomorphism in $\Derb(\C_\Xsa)$. 
\end{theorem}
\begin{proof}
Let $U\in \Op_\Xsa$. The sheaves $\Cinftp_\Xsa$ and $\Dbt_\Xsa$ being $\sect$-acyclic, 
$\rsect(U;\Ot_\Xsa)$ and $\rsect(U;{\tw\OO}^\tp_\Xsa)$ are represented by the complexes
\eqn\label{eq:dolbOt1U}
\rsect(U;\Ot_\Xsa)&\cl&0\to \shc_\Xsa^{\infty,\tp (0,0)}(U)\to[\ol\partial]\cdots\to[\ol\partial]\shc_\Xsa^{\infty,\tp (0,n)}(U)\to0,\\
\rsect(U;\tw\OO^\tp_\Xsa)&\cl&0\to \Db_\Xsa^{\tp(0,0)}(U)\to[\ol\partial]\cdots\to[\ol\partial]\Db_\Xsa^{\tp(0,n)}(U)\to0.
\eneqn
When $X=\C^n$, it is proved in~\cite{KS96}*{Th.~10.5} that these two complexes are quasi-isomorphic and this is enough for our purpose  since the statement is of local nature.
\end{proof}

\subsubsection*{Whitney  holomorphic functions}
One also defines the (derived) sheaf of Whitney  holomorphic functions by taking the 
Dolbeault complex of the sheaf  $\Cinfw_\Msa$ 
\eq\label{eq:dolbOt1}
\Ow_\Xsa&\eqdot&0\to \shc_\Xsa^{\infty, {\rm w}(0,0)}\to[\ol\partial]\cdots\to[\ol\partial]\shc_\Xsa^{\infty,{\rm w} (0,n)}\to0.
\eneq

Following~\cite{KS96}, we shall use the quasi-abelian  categories of 
$\FN$ or $\DFN$   spaces (see~\cite{Sn99}). 
The  topological duality functor induces an equivalence 
of triangulated categories
\eqn
&&\Derb(\FN)^\rop\simeq \Derb(\DFN).
\eneqn

By applying Proposition~\ref{pro:KS9622}, one gets:
\begin{proposition}\label{pro:KSduality}
Let $U\in\Op_\Xsa$. 
The two objects $\rsect(U;\Ow_\Xsa)$ and $\rsect(U;\Ot_\Xsa)$ are well-defined in the categories 
$\Derb(\FN)$ and $\Derb(\DFN)$ respectively, and are dual to each other.
\end{proposition}
See~\cite[Theorem 6.1]{KS96} for a more general statement.

\begin{example}\label{exa:OtOwCDb}
(i) Let $Z$ be a closed complex analytic subset of the complex manifold $X$. We have the isomorphisms in $\Derb(\D_X)$:
\eqn
&&
\ba{ll}
%\opb{\rho_M}\rhom[{\C_\Xsa}](\RD_X'\C_Z,\Oww_\Xsa)\simeq(\OO_X)_Z&\mbox{(restriction)},\\[1ex]
\opb{\rho_X}\rhom[{\C_\Xsa}](\RD_X'\C_Z,\Ow_\Xsa)\simeq\OO_X\widehat{\vert}_Z&
\mbox{(formal completion)},\\[1ex]
\opb{\rho_X}\rhom[{\C_\Xsa}](\C_Z,\Ot_\Xsa)\simeq\rsect_{[Z]}(\OO_X)&\mbox{(algebraic cohomology)}. %,\\[1ex]
%\opb{\rho_X}\rhom[{\C_\Xsa}](\C_Z,\OO_\Xsa)\simeq\rsect_{Z}(\OO_X)&\mbox{(local cohomology)}.
\ea
\eneqn
(ii) Let  $M$ be a real analytic manifold such that $X$ is a complexification of $M$. We have the isomorphisms  in $\Derb(\D_M)$:
\eqn
&&\ba{ll}
%\opb{\rho_X}\rhom[{\C_\Xsa}](\RD_X'\C_M,\Oww_\Xsa)\vert_M\simeq \rA_M &\mbox{(real analytic functions)},
%\\[1ex]
\opb{\rho_X}\rhom[{\C_\Xsa}](\RD_X'\C_M,\Ow_\Xsa)\vert_M\simeq \Cinf_M&\mbox{($\mathrm{C}^\infty$-functions)},
\\[1ex]
\opb{\rho_X}\rhom[{\C_\Xsa}](\RD_X'\C_M,\Ot_\Xsa)\vert_M\simeq\Db_M&\mbox{(distributions)}. %,\\[1ex]
%\opb{\rho_X}\rhom[{\C_\Xsa}](\RD_X'\C_M,\OO_\Xsa)\vert_M\simeq\shb_M&\mbox{(hyperfunctions)}.
\ea
\eneqn
\end{example}

\section{The vanishing theorem}

\subsection{Stein subanalytic sets}

The next lemmas will be useful in the sequel.

Recall that a compact subset $K$ of a complex manifold is said to be Stein if $K$ admits a fundamental neighborhood system consisting of Stein open subsets. 

\begin{lemma}\label{le:obv1}
Let $X$ be a Stein manifold and let $K$ be a compact subset. Then there exist an open  Stein subanalytic subset $U$ of $X$ 
and a  compact Stein subanalytic subset $L$ of $X$  with $K\subset L\subset U$. 
\end{lemma}
\begin{proof}
One embeds $X$ as a  closed smooth complex submanifold of $\C^N$ for some $N$. Then choose for $U$ the intersection of $X$ with an open ball which contains $K$ and similarly choose a closed ball for $L$. 
\end{proof}
\begin{lemma}\label{le:obv2}
Let $X$ be a complex manifold and let $K$ be a Stein compact subset. Then there exists a fundamental neighborhood  system of $K$ 
consisting of open Stein subanalytic subsets as well a a fundamental neighborhood  system of $K$ consisting of compact Stein subanalytic subsets. 
\end{lemma}
\begin{proof}
Choose a Stein open neighborhood $W$ of $K$ and apply Lemma~\ref{le:obv1}.
\end{proof}

 The proof of  Lemma~\ref{le:barlet} below was  suggested to us by Daniel Barlet. Note that this lemma can be compared to~\cite{BM11} which treats complex neighborhoods of real manifolds.

\begin{lemma}\label{le:barlet}
Let $X$ be a closed smooth Stein submanifold of $Y=\C^N$ and let $U\subset X$ be an subanalytic relatively compact Stein open subset of $X$. Then for each open neighborhood $W$ of $X$ in $Y$, there exists a subanalytic relatively compact  Stein  open subset $V$ of $W$ such that $V\cap X=U$.
\end{lemma}
\begin{proof}
Denote by $T_XY$ the normal bundle to $X$ in $Y$ and identify $X$ with the zero-section of $T_XY$. 
By~\cite{GR65}*{Ch.~8, Sect.~C,Th.~8}, there exist an open neighborhood $\Omega$ of $X$ in $Y$, an open neighborhood $\tw \Omega$ of $X$ in $T_XY$ and a holomorphic isomorphism $\Omega\isoto \tw \Omega$. It is thus enough to construct $V$ in $\tw \Omega\subset T_XY$. 
Since $U\times_XT_XY$ is a vector bundle over a Stein manifold, it is Stein by loc.\ cit.\ Th.~9. 
Moreover, since $T_XY$ is Stein, there exists an open relatively compact subanalytic subset $W$ of $T_XY$ which contains $\ol U$. 
Recalling that $\R^+$ acts on the vector bundle $T_XY$ we get that 
the open set $V=(c\cdot W)\cap(U\times_XT_XY)$ is Stein, subanalytic and contained in $\tw \Omega$ for $c>0$ small enough. 
\end{proof}

Let $K$ be a Stein subanalytic compact subset of a complex manifold $X$ and consider the ring $A\eqdot\sho_X(K)$. 
Let  $\mdc[\sho_X\vert_K]$ denote the category of coherent $\sho_X$-modules  defined in a neighborhood of $K$ and let 
 $\mdf[A]$ denote the category of finitely generated $A$-modules.  It is well-known after the work of~Frisch and~Siu
(see~\cite{Fr67, Si69} and~\cite{Ta02}*{Th.~11.9.2}) that  $A$ is Noetherian and that the functor $\sect(K;\scbul)$ induces 
 an equivalence of categories
\eq\label{eq:eqv}
&&\mdc[\sho_X\vert_K]\isoto \mdf[A]. 
\eneq

\begin{lemma}\label{le:obv4}
Let $K$ be a Stein subanalytic subset of the complex manifold $X$. 
The category  $\mdc[\sho_X\vert_K]$ has finite homological dimension. Moreover, any object of this category admits a finite  resolution by projective objects and projective objects are direct factors of finite free $\sho_X$-modules.
\end{lemma}
\begin{proof}
The functor $\hom[\sho_X]$ has finite homological dimension and the functor $\sect(K;\scbul)$ is exact. Therefore, the functor 
$\Hom[\sho_X]$ has finite homological dimension.
This proves the first assertion. The second one follows from the equivalence~\eqref{eq:eqv}.
\end{proof}

\subsection{A vanishing theorem on the affine space}
\subsubsection*{The sheaf $\OO_\Xsa^{2,\tp}$}

Recall the spaces $L^{2,s}(U)$ and $L^{2,\tp}(U)$  of~\eqref{eq:L2sU} and~\eqref{eq:L2tpU}. On a complex manifold we denote by  $L^{2,s,(p,q)}(U)$ and $L^{2,\tp,(p,q)}(U)$   the spaces of differential  forms with coefficients in these spaces. We set
\eq\label{eq:L20s}
&&\ba{l}
L^{2,s,(p,q)}_0(U)\eqdot\{f\in L^{2,s,(p,q)}(U);\ol\partial f\in L^{2,s,(p,q+1)}(U)\},\label{eq:L2s0}\\
L_{0}^{2,\tp,(p,q)}(U)\eqdot\{f\in L^{2,\tp,(p,q)}(U);\ol\partial f\in L^{2,\tp,(p,q+1)}(U)\},\\
L_{0}^{2,\tpst,(p,q)}(U)\eqdot\indlim[s] L^{2,s,(p,q)}_0(U).
\ea\eneq

\begin{lemma}\label{le:dolbOt2tpst}
The natural morphism $L_0^{2,\tpst,(p,q)}(U) \to L_0^{2,\tp,(p,q)}(U) $ is an isomorphism. 
\end{lemma}
The proof is obvious but for the reader's convenience, we develop it. 
\begin{proof}
Set for short 
\eqn
&&E^s=L^{2,s,(p,q)}(U),\quad E^\tp=\sindlim E^s,\quad
F^s=L^{2,s,(p,q+1)}(U),\quad F^\tp=\sindlim F^s,\\   
&&G=\Db^{\tp,(p,q+1)}(U),\quad u=\ol\partial\cl E^\tp\to G, \\
&&E_0^s=\{x\in E^s;u(x)\in F^s\},\quad E_0^\tpst=\sindlim E_0^s,\quad E^\tp_0=\{x\in E^\tp;u(x)\in F^\tp\}.
\eneqn
Notice that we have monomorphisms $E^s\into E^t$ and $F^s\into F^t$ for $s\leq t$.
The morphism $E_0^\tpst\to E_0^\tp$ is a monomorphism since both spaces are contained in $E^\tp$. Let us show that it is an epimorphism.
Let $x\in E_0^\tp$. There exists $s$ and $t\geq s$  such that $x\in E^s$ and $u(x)\in F^t$. Therefore, $x\in E_0^t$. 
\end{proof}

We consider the complexes
\eq
%&&\ba{l}
\OO^{2,s}(U)&\eqdot&0\to L^{2,s,(0,0)}_0(U) \to[\ol\partial]\cdots\to[\ol\partial]L^{2,s,(0,n)}_0(U)\to0,\label{eq:dolbOt2s}\\
\OO^{2,\tp}(U)&\eqdot&0\to L_{0}^{2,\tp,(0,0)}(U) \to[\ol\partial]\cdots\to[\ol\partial]L_{0}^{2,\tp,(0,n)}(U)\to0.\label{eq:dolbOt2tp}
%\ea
\eneq

We shall first recall a fundamental result due to H\"ormander. 

\begin{theorem}[{see~\cite{Ho65}*{Th.~2.2.1'}}]\label{th:horm}
Assume that $X=\C^n$ and $U\subset X$ is a relatively compact open subset and is Stein. Then the complex~\eqref{eq:dolbOt2s} is concentrated in degree $0$.
\end{theorem}
Note that H{\"o}rmander's theorem applies since the function  $\phi\eqdot-\ln(\dist(x,X\setminus U))$ is plurisubharmonic on the Stein open subset $U$.

\begin{example}
Let $X=A^1(\C)$, the complex line with coordinate $z$, and let $U=D^\times$ the disc minus $\{0\}$. 
The map $\ol\partial\cl L_0^{2,1}(U)\to L^{2,1}(U)$ is  surjective (note that  $L^{2,1}(U)\simeq L_0^{2,1,(0,1)}(U)$). For example, $1/\ol z\in  L^{2,1}(U)$ and 
$\ln\vert z\vert\in  L^{2,1}(U)$ is  a solution of the equation $ \ol\partial u=1/\ol z$.  
\end{example}

\begin{corollary}\label{cor:horm}
Assume that $X=\C^n$ and $U\in\Op_\Xsa$  is Stein. Then the complex~\eqref{eq:dolbOt2tp}
is concentrated in degree $0$.
\end{corollary}
\begin{proof}
By Lemma~\ref{le:dolbOt2tpst}, one has $\sindlim L_{0}^{2,s,(p,q)}(U)\isoto L_{0}^{2,\tp,(p,q)}(U)$. Since the inductive limit is filtrant, it commutes with the functor of cohomology and the result follows from~Theorem~\ref{th:horm}. 
\end{proof}
Denote by  $\shl_{\Xsa}^{2,\tp,(p,q)}$ and $\shl_{0,\Xsa}^{2,\tp,(p,q)}$ the presheaves $U\mapsto L^{2,\tp,(p,q)}(U)$ and $U\mapsto L_{0}^{2,\tp,(p,q)}(U)$ on $\Xsa$, respectively.

It follows from Lemma~\ref{le:L2tpshv} that the presheaves $\shl_{\Xsa}^{2,\tp,(p,q)}$ are sheaves of  $\Cinftp_\Xsa$-modules.

\begin{lemma}\label{le:L2tpacyc}
The presheaves  $\shl_{0,\Xsa}^{2,\tp,(p,q)}$  are sheaves and are $\Cinftp_\Xsa$-modules.  In particular, these sheaves are $\sect$-acyclic.
\end{lemma}
\begin{proof}
 The fact that  $\shl_{0,\Xsa}^{2,\tp,(p,q)}$ is a sheaf follows from the fact that  $\shl_{\Xsa}^{2,\tp,(p,q+1)}$ is a sheaf.
It is a $\Cinftp_\Xsa$-module since $L^{2,\tp,(p,q)}(U)$ is a $ \Cinftp_\Xsa(U)$-module and for 
$\theta\in \Cinftp_\Xsa(U)$ and $\phi\in  L_{0}^{2,\tp,(p,q)}(U)$, 
$\ol\partial(\theta\phi)=(\ol\partial\theta)\phi+\theta(\ol\partial\phi)$.
\end{proof}

We define the object $\OO_\Xsa^{2,\tp}\in\Derb(\C_\Xsa)$ by the complex of sheaves on $\Xsa$:
\eq\label{eq:dolbOt2tp2}
\OO_\Xsa^{2,\tp}&\eqdot&0\to \shl_{0,\Xsa}^{2,\tp,(0,0)} \to[\ol\partial]\cdots\to[\ol\partial]\shl_{0,\Xsa}^{2,\tp,(0,n)}\to0.
\eneq
It follows from Lemma~\ref{le:L2tpacyc} that, for $U\in\Op_\Xsa$, the object
$\rsect(U;\OO_\Xsa^{2,\tp})$ is represented by the complex~\eqref{eq:dolbOt2tp}

\begin{theorem}\label{th:main1}
Assume that $X=\C^n$ and $U\in\Op_\Xsa$  is Stein.  Then $\rsect(U;\Ot_\Xsa)$ is concentrated in degree $0$.
\end{theorem}
\begin{proof}
It follows from Proposition~\ref{pro:l2Ctpmod} that there are natural morphisms in $\Derb(\C_\Xsa)$
\eqn
&&\Ot_\Xsa\to[\alpha]\ \OO_\Xsa^{2,\tp}\to[\beta]\tw\OO^\tp_\Xsa
\eneqn
which induce
\eqn
&&\rsect(U;\Ot_\Xsa)\to[\alpha]\rsect(U;\OO_\Xsa^{2,\tp})\to[\beta]\rsect(U;\tw\OO^\tp_\Xsa).
\eneqn
The composition $\beta\circ\alpha$ is an isomorphism by Theorem~\ref{th:ks96}
and the cohomology of the complex $\rsect(U;\OO_\Xsa^{2,\tp})$ is concentrated in degree $0$ by Corollary~\ref{cor:horm}. The result follows. 
\end{proof}

Recall the functor $\eim{\rho_X}$ of~\eqref{eq:eimrho} and  consider a coherent $\sho_X$-module $\shf$. We define the sheaf
\eq\label{eq:shftp}
&&\shf^\tp\eqdot \eim{\rho_X}\shf\ltens[\eim{\rho_X}\sho_X]\Ot_\Xsa.
\eneq

\begin{corollary}\label{cor:cohvanish1}
Assume that $X=\C^n$ and $U\in\Op_\Xsa$  is Stein. 
 Let $\shf$ be a coherent  $\sho_X$-module defined in a neighborhood of a Stein compact subset $K$ of $X$ such that $U\subset K$. Then 
 $\rsect(U;\shf^\tp)$ is concentrated in degree $0$.
 \end{corollary}
\begin{proof}
By Lemma~\ref{le:obv2}, we may assume that  $K$ is a compact Stein subanalytic subset of $X$. 
By Lemma~\ref{le:obv4} there exists an exact sequence 
\eq\label{eq:Kresolution1}
&&0\to\shl_p\to\cdots\to\shl_0\to\shf\to0
\eneq
where   all $\shl_i$ are projective objects of the category $\mdc[\sho_X\vert_K]$, hence  direct factors of finite  free $\sho_X$-modules. Since the functor $\eim{\rho_X}$ is exact, 
 $\eim{\rho_X}\shf$ is quasi-isomorphic to the  complex
 \eqn
&&0\to\eim{\rho_X}\shl_p\to\cdots\to\eim{\rho_X}\shl_0\to0
\eneqn
 and thus $\shf^\tp$ is  quasi-isomorphic to the complex
\eq\label{eq:Kresolution2}
&&0\to\shl_p^\tp\to\cdots\to\shl_0^\tp\to0.
\eneq
Note  that for each $i$, $\shl_i$ being a direct factor of of finite  free $\sho_X$-modules, 
 $\shl_i^\tp$ is concentrated in degree $0$. 

We shall argue by induction on $p$. If $p=0$ the result follows from Theorem~\ref{th:main1}.  Now assume that the result holds for any coherent sheaf which admits a projective of length $\leq p-1$. 
 Define $\shg$ by the exact sequence
\eqn
&&0\to\shl_p\to\cdots\to\shl_1\to\shg\to0.
\eneqn
Then $\rsect(U;\shg^\tp)$ is concentrated in degree $0$ by the induction hypothesis. Moreover, 
$\shg^\tp$ is  quasi-isomorphic to the complex
\eq\label{eq:Kresolution2G}
&&0\to\shl_p^\tp\to\cdots\to\shl_1^\tp\to0.
\eneq
It follows that  $0\to\shg^\tp\to\shl_0^\tp\to\shf^\tp\to0$ is an exact sequence of sheaves on $\Xsa$ 
and the result follows from the long exact sequence obtained  by applying the functor $\sect(U;\scbul)$. 
\end{proof}

\begin{remark}
On a complex manifold, it would be possible to  replace the subanalytic topology with the Stein subanalytic topology for which the open sets are the finite union of 
Stein relatively compact subanalytic open subsets of $X$ (see~\cite{Pe17}).
With this new topology the sheaf of  holomorphic functions and, thanks to Theorem~\ref{th:main1},
 the sheaf of  temperate
 holomorphic functions,  are concentrated in degree $0$. 
 \end{remark}

 \subsection{A vanishing theorem on Stein manifolds}
In this section, we shall extend Theorem~\ref{th:main1}  by replacing $\C^n$ with a complex manifold. 
\begin{lemma}\label{le:XtpinYtp}
Let $j\cl X\into Y$ be a closed embedding of smooth complex manifolds. Then there is a natural isomorphism 
$(\oim{j}\sho_X)^\tp\isoto\oim{j}\sho_\Xsa^\tp$.
\end{lemma}
\begin{proof}
By Lemma~\ref{le:eimrhooimf}, we have  the isomorphism of functors
\eqn
&&\opb{j}\eim{\rho_Y}\simeq \eim{\rho_X}\opb{j}
\eneqn
which induces the morphisms
  \eqn%\label{eq:YtptoXtp2}
 \opb{j} (\oim{j}\sho_X)^\tp&\simeq&\opb{j} \eim{\rho_Y}\oim{j}\sho_X\ltens[\opb{j}\eim{\rho_Y}\sho_Y]\opb{j}\Ot_\Ysa\\
  &\simeq& \eim{\rho_X}\opb{j}\oim{j}\sho_X\ltens[\opb{j}\eim{\rho_Y}\sho_Y]\opb{j}\Ot_\Ysa
  \to\opb{j}\Ot_\Ysa.
 \eneqn
 Here,   the last morphism is associated with  $\opb{j}\eim{\rho_Y}\sho_Y\simeq \eim{\rho_X}\opb{j}\sho_Y\to  \eim{\rho_X}\sho_X$.
 
 On the other hand, there is a natural morphism $\opb{j}\Cinftp_\Ysa \to\Cinftp_\Xsa$ which induces the morphism
 $\opb{j}\sho_\Ysa^\tp \to\sho_\Xsa^\tp$. These define the morphism
 \eq\label{eq:YtptoXtp1}
 &&(\oim{j}\sho_\Xsa)^\tp \to\oim{j}\sho_\Xsa^\tp.
 \eneq
Let us prove that~\eqref{eq:YtptoXtp1} is an isomorphism. This is a local problem and we may assume that $Y=X\times Z$ with 
$X=\C^n$, $Z=\C^p$ and $j$ is the embedding identifying $X$ and $X\times\{0\}$. By induction we may assume $p=1$. Let $z$ denote a holomorphic coordinate on $Z=\C$. For simplicity, we do not write $X$. Then we are reduced to prove that the complex 
$\sho^\tp_\Ysa\to[z]\sho^\tp_\Ysa$ is quasi-isomorphic to $\C_{\{0\}}[1]$. Let us replace $\sho^\tp_\Ysa$ with the complex
$\shc_\Ysa^{\infty,\tp}\to[\ol\partial]\shc_\Ysa^{\infty,\tp}$ we get the double complex 
\eq\label{eq:zolpartial}
&&\xymatrix{
\shc_\Ysa^{\infty,\tp}\ar[r]^-{\ol\partial}\ar[d]^-z&\shc_\Ysa^{\infty,\tp}\ar[d]^-z\\
\shc_\Ysa^{\infty,\tp}\ar[r]^-{\ol\partial}&\shc_\Ysa^{\infty,\tp}.
}\eneq
Given $U,V\in\Op_\Ysa$ with  $U\subset V$ and $0\in U$,  the two complexes 
$\sect(W;\shc_\Ysa^{\infty,\tp})\to[z]\sect(W;\shc_\Ysa^{\infty,\tp})$ with $W$ being either $U$ or $V$ are quasi-isomophic. 
Hence, in order to calculate the cohomology of the double complex~\eqref{eq:zolpartial} we may apply to it the functor $\sect(U;\scbul)$ for any   $U\in\Op_\Ysa$ with $0\in U$. If we choose $U$ convex, this complex is quasi-isomorphic to the complex 
$\sho^\tp_\Ysa(U)\to[z]\sho^\tp_\Ysa(U)$ which is itself quasi-isomorphic to $\C_{\{0\}}[1]$. 
\end{proof}
\begin{lemma}\label{le:XtpinYtpB}
Let $j\cl X\into Y$ be a closed embedding of smooth complex manifolds and let $\shf$ be a coherent $\sho_X$-module. Then there is a natural isomorphism 
$(\oim{j}\shf)^\tp\isoto\oim{j}\shf^\tp$.
\end{lemma}
\begin{proof}
One constructs the natural morphism 
 \eq\label{eq:YtptoXtp1B}
 &&(\oim{j}\shf)^\tp \to\oim{j}\shf^\tp
 \eneq
 by the same procedure as for $\sho_X$ in~\eqref{eq:YtptoXtp1}. To prove that this morphism is an isomorphism, one may replace locally 
 $\shf$ with a free resolution. 
\end{proof}

\begin{theorem}\label{th:main2}
Let $X$ be a complex Stein manifold and let $U$ be a subanaytic relatively compact Stein open subset of $X$ contained in a 
Stein compact subset $K$ of $X$. 
 Let $\shf$ be a coherent $\sho_X$-module defined in a neighborhood of $K$. Then $\rsect(U;\shf^\tp)$ is concentrated in degree $0$.
\end{theorem}
\begin{proof}
(i) Since $X$ is Stein, there exist some integer $N$ and a closed embedding $j\cl X\into \C^N$. Set $Y=\C^N$ for short. 

\spa
(ii)  The coherent $\sho_Y$-module $\oim{j}\shf$ is defined in a neighborhood of $K$ in $Y$ and $K$ admits a fundamental neighborhood system of  Stein open subsets in $Y$ by~\cite{Si76}.  Let $W$ be such a Stein open subset on which $\oim{j}\shf$ is defined. 

\spa
(iii) By applying Lemma~\ref{le:barlet}, we find  a relatively compact subanalytic Stein open subset $V$ of $Y$ such that 
 $U=X\cap V$. By replacing $V$ with $V\cap V'$ for a Stein open subanalytic subset containing $\ol U$, we may assume that 
 $\ol V\subset W$. 

 \spa
(iv) Applying the result of Corollary~\ref{cor:cohvanish1}, we get that
 \eq\label{eq:van2}
 &&\rsect(V;(\oim{j}\shf)^\tp)\mbox{ is concentrated in degree $0$.}
 \eneq 
 Applying Lemma~\ref{le:XtpinYtpB}, we get 
  \eq\label{eq:van3}
 &&\rsect(V;\oim{j}\shf^\tp)\mbox{ is concentrated in degree $0$.}
 \eneq 
 Since $\rsect(V;\oim{j}\shf^\tp)\simeq\rsect(U;\shf^\tp)$, the proof is complete.
  \end{proof}
  
\begin{remark}
Theorem~\ref{th:main1}  was deduced from H\"ormander's Theorem~\ref{th:horm} and the same argument would apply on a complex manifold if the  H\"ormander's theorem had been stated in such a framework. And indeed, according to H.~Skoda, such a generalization of H\"ormander's theorem should be  possible when combining~\cite{De18}*{Ch.~VIII~\S~6, Th.~6.5} and~\cite{El75}. This would provide an alternative proof to Theorem~\ref{th:main2}.
\end{remark}

\begin{corollary}\label{cor:vanishw}
Let $X$ be a complex Stein manifold of pure dimension $n$ and let $U$ be a subanaytic relatively compact Stein open subset of $X$. Then 
$\rsect(U;\Ot_\Xsa)$ is concentrated in degree $0$ and $\rsect(U;\Ow_\Xsa)$ is concentrated in degree $n$.
\end{corollary}
\begin{proof}
The first vanishing result is a particular case of Theorem~\ref{th:main2} and the second result 
 follows by applying Proposition~\ref{pro:KSduality}.
  \end{proof}

\providecommand{\bysame}{\leavevmode\hbox to3em{\hrulefill}\thinspace}

\noindent
\parbox[t]{21em}
{\scriptsize{
\noindent
Pierre Schapira\\
Sorbonne Universit{\'e}, CNRS IMJ-PRG\\
4 place Jussieu, 75252 Paris Cedex 05 France\\
e-mail: pierre.schapira@imj-prg.fr\\
http://webusers.imj-prg.fr/\textasciitilde pierre.schapira/
}}

\end{document}